\documentclass[a4paper]{amsart}
\usepackage[new]{old-arrows}
\usepackage[T1]{fontenc}
\usepackage[english]{babel}
\usepackage{kpfonts}

\usepackage{amsmath,amssymb,amsthm,enumerate,xspace}
\usepackage{comment}
\usepackage{accents}
\usepackage[colorlinks=true,citecolor=blue,linkcolor=blue]{hyperref}
\usepackage[all]{xy}
\usepackage{tikz}
\usepackage{mathrsfs}
\usepackage{cleveref}
\usepackage{graphicx} 

\numberwithin{equation}{section}

\newtheorem{thm}{Theorem}[section]
\newtheorem{cor}[thm]{Corollary}

\newtheorem{lem}[thm]{Lemma}

\newtheorem{quest}[thm]{Question}
\newtheorem{claim}[thm]{Claim}

\newtheorem*{openproblem*}{Problem}
\newtheorem*{quest*}{Question}
\newtheorem*{problem*}{Problem}
\newtheorem*{claim*}{Claim}

\theoremstyle{definition}

\theoremstyle{remark}
\newtheorem{rem}[thm]{Remark}

\newcommand{\bD}{\mathbb{D}}

\newcommand{\bQ}{\mathbb{Q}}
\newcommand{\bR}{\mathbb{R}}
\newcommand{\bS}{\mathbb{S}}

\newcommand{\bZ}{\mathbb{Z}}

\newcommand\Diff{\mathrm{Diff}}

\newcommand\BDiff{\mathrm{BDiff}}
\newcommand\dDiff{\mathrm{Diff}^{\delta}}
\newcommand\drDiff{\mathrm{Diff}^{r,\delta}}

\newcommand\BdDiff{\mathrm{BDiff}^{\delta}}

\newcommand{\hcoker}{/\!\!/}

\makeatother
\addtocontents{toc}{\protect\setcounter{tocdepth}{1}}
\makeatletter
\let\c@equation\c@thm
\makeatother
\numberwithin{equation}{section}
\makeatletter
\let\@wraptoccontribs\wraptoccontribs
\makeatother
\title[]{On invariants of foliated sphere bundles}

\author[]{Sam Nariman}
\address{Department of Mathematics\\
  Purdue University\\
150 N. University Street\\
West Lafayette, IN 47907-2067\\
}
\email{snariman@purdue.edu}

\begin{document}
\begin{abstract}
Morita (\cite{morita1984nontriviality}) showed that for each integer $k\geq 1$,  there are examples of flat $\bS^1$-bundles for which the $k$-th power of the Euler class does not vanish. Haefliger asked (\cite[page 154]{MR501011}) if the same holds for flat odd-dimensional sphere bundles. In this paper, for a manifold $M$ with a free torus action, we prove that certain $M$-bundles are cobordant to a flat $M$-bundle and as a consequence, we answer Haefliger's question. We show that all monomials in the Euler class and Pontryagin classes $p_i$ for $i\leq n-1$ are non-trivial in $H^*(\BdDiff_+(\bS^{2n-1});\bQ)$. 
%In the appendix, Nils Prigge corrects a claim by Haefliger (\cite[Page 154]{MR501011}) about the vanishing of certain classes in the smooth group cohomology of $\Diff_+(\bS^3)$. 
\end{abstract}
\maketitle
\section{Introduction}
 Let $\Diff_+(\bS^d)$ be the group of orientation-preserving smooth diffeomorphisms of the sphere $\bS^d$. Haefliger (\cite[page 242, problem 4]{MR0494114}) asked whether, for a given integer $k>1$, there exists a manifold $M$ and a representation $\pi_1(M)\to \Diff_+(\bS^1)$ such that the $k$-th power of the Euler class of the associated flat circle bundle is non-trivial.  
 
 For $k=1$,  Benzecri (\cite{MR2612352}) and Milnor (\cite{MR95518}) constructed flat linear $\bR^2$-bundles over surfaces with non-trivial Euler classes. As a consequence, there are associated flat circle bundles to the examples of Benzecri and Milnor with non-trivial Euler classes. Morita in \cite{morita1984nontriviality} answered Haefliger's question affirmatively by proving a more general theorem (see also \cite{nariman2016powers}). In particular, he showed that the powers of the Euler class $e^k$ as a class in $H^{2k}(\BdDiff_+(\bS^1);\bQ)$ are non-trivial for each integer $k$, where $\BdDiff_+(\bS^1)$ denotes the classifying space of orientable flat $\bS^1$-bundles ($\delta$ denotes the discrete topology on a given topological group).

 % Let $\BdDiff_+(\bS^n)$ denote the classifying space of orientable flat $\bS^n$-bundles ($\delta$ denotes the discrete topology on a given topological group) and let $\mathcal{L}_{\bS^n}$ denote the topological Lie algebra of smooth vector fields on $\bS^n$. There is a natural map
% \[
%\Phi\colon H^*(\mathcal{L}_{\bS^1}, \text{so}(2))\to H^*(\BdDiff_+(\bS^1);\bR),
% \]
% where $H^*(\mathcal{L}_{\bS^1}, \text{so}(2))$ is the relative continuous Lie algebra cohomology (see \cite[Page 44]{haefliger2010differential}). The Lie algebra cohomology $H^*(\mathcal{L}_{\bS^1}, \text{so}(2))$ is isomorphic to $\bR[e, \text{gv}]/(e\cdot \text{gv}=0)$ where $e$ is the Euler class and $\text{gv}$ is also a degree $2$ class known as the Godbillon-Vey class. Morita (\cite[Theorem 1.1]{morita1984nontriviality}) showed that this map is injective. 
% 
% Later Haefliger \cite{MR501011} used rational homotopy theory models to study the image of 
% \[
% H^*(\mathrm{B}\text{SO}(n+1); \bR)\to H^*(\mathcal{L}_{\bS^n}, \text{so}(n+1)),
% \]
% and he realized that for $n$ odd the image of the powers of the Euler class in $H^*(\mathcal{L}_{\bS^n}, \text{so}(n+1))$ are non-trivial. Hence, he posed a more general version of his previous question (\cite[Page 154]{MR501011}).
Haefliger in \cite{MR501011} used rational homotopy theory to study the continuous Lie algebra cohomology of the Lie algebra of vector fields on $\bS^d$ relative to $\text{SO}(d+1)$. His calculations suggested that the powers of the Euler class might be non-trivial in the smooth group cohomology of the diffeomorphism group of odd-dimensional spheres. So he posed the following question  (\cite[page 154]{MR501011}).
 \begin{quest}[Haefliger] Are there flat $(2n - 1)$-sphere bundles
with a non-zero power of the Euler class?
 \end{quest}
To answer this question affirmatively, we prove a more general non-vanishing result about monomials in the Euler class $e$ and the Pontryagin classes $p_i$.
 \begin{thm}\label{oddspheres}
 All the monomials in $e$ and $p_i$ for $i\leq n-1$ are non-trivial in $H^*(\BdDiff_+(\bS^{2n-1});\bQ)$.
 \end{thm}
 \begin{rem}
 Morita pointed out to the author that the involution on $\bS^{2n-1}$ induced by reflecting through a hyperplane changes the sign of the Euler class but it does not change the Pontryagin classes. Hence, the classes $e^k$ and monomials in Pontryagin classes are linearly independent when $k$ is odd. However, the author does not know whether in general the monomials are linearly independent in $H^*(\BdDiff_+(\bS^{2n-1});\bQ)$. 
 \end{rem}
 The techniques that we use also work for volume-preserving diffeomorphisms $\Diff_{\text{vol}}(\bS^{2n-1})$, so these classes  are also non-trivial in $H^*(\BdDiff_{\text{vol}}(\bS^{2n-1});\bQ)$.
% \begin{rem}\label{claim}
%There is a van Est type theorem (\cite[Page 43]{haefliger2010differential}) that implies that $H^*(\mathcal{L}_{\bS^3}, \text{so}(4))$ is isomorphic to the smooth group cohomology $H^*_{\text{sm}}(\Diff_+(\bS^3);\bR)$.  In higher dimensions, it is expected from Bott's belief in \cite[Page 217]{MR488080} that there is at least a map
%\[
%H^*_{\text{sm}}(\Diff_+(\bS^n);\bR)\to H^*(\mathcal{L}_{\bS^n}, \text{so}(n+1)).
%\]
%Haefliger in \cite[Page 154]{MR501011} sketched his method to prove a claim  that the kernel of the map 
%  \[
% H^*(\mathrm{B}\text{SO}(n+1); \bR)\to H^*(\mathcal{L}_{\bS^n}, \text{so}(n+1)),
% \]
% is generated by the monomials in Pontryagin classes $p_1, \dots, p_{\lfloor n/2\rfloor}$ whose degrees are larger than $2n$. So according to his claim, $p_1^2$ vanishes in $H^*_{\text{sm}}(\Diff_+(\bS^3);\bR)$ and as a consequence it would vanish also in $H^*(\BdDiff_+(\bS^{3});\bR)$. But in the appendix by Nils Prigge, we shall see that Haefliger's vanishing works when $n$ is even, and using his method we shall see for $n=3$, in fact, the class $p_1^2$ is not zero in the smooth group cohomology of $\Diff_+(\bS^3)$ which is isomorphic to $H^*(\mathcal{L}_{\bS^3}, \text{so}(4))$.
% \end{rem}
 
 Let us recall how the Pontryagin classes are defined in $H^*(\BdDiff_+(\bS^{2n-1});\bQ)$. There is a natural map 
 \[
\eta\colon  \BdDiff_+(\bS^{2n-1})\to \BDiff_+(\bS^{2n-1})
 \]
 that is induced by the identity homomorphism $\dDiff_+(\bS^{2n-1})\to \Diff_+(\bS^{2n-1})$. The homotopy fiber of $\eta$ is denoted by $\overline{\BDiff_+(\bS^{2n-1})}$. By coning the sphere to a disk and then restricting it to the interior of the disk, we obtain the following maps
 \begin{equation}\label{cone}
 \BDiff_+(\bS^{d})\to \mathrm{BHomeo}_+(\bD^{d+1})\to \mathrm{BHomeo}_+(\bR^{d+1}).
 \end{equation}
 There are topological Pontryagin classes for Euclidean fiber bundles that are defined rationally in $H^*(\mathrm{BHomeo}_+(\bR^{d+1});\bQ)$ (see \cite{MR4372633}).  Galatius and Randal-Williams (\cite{galatius2022algebraic}) proved a remarkable result that the map
 \[
 \bQ[e,p_1, p_2, \dots]\to H^*(\mathrm{BHomeo}_+(\bR^{2n});\bQ)
 \]
 is injective for $2n\geq 6$. They also proved that the map 
 \[
 \bQ[e,p_1, p_2, \dots]\to H^*(\BDiff_+(\bS^{2n-1}); \bQ),
 \]
induced by pulling back these classes to $H^*(\BDiff_+(\bS^{2n-1}); \bQ)$, is injective for $2n-1\geq 9$. To prove the non-triviality of these classes for odd-dimensional sphere bundles (not necessarily flat sphere bundles), they constructed certain finite group actions on spheres. As a result of their method, they also obtained that for $2n-1\geq 5$, the map
 \[
  \bQ[e,p_1, p_2, \dots]\to H^*(\BdDiff_+(\bS^{2n-1}); \bZ)\otimes \bQ
 \]
 is injective. However, their method of using finite group actions does not say if the image of the composition 
 \[
  \bQ[e,p_1, p_2, \dots]\to H^*(\BdDiff_+(\bS^{2n-1}); \bZ)\otimes \bQ\to H^*(\BdDiff_+(\bS^{2n-1}); \bQ),
 \]
 is non-trivial. 
 \begin{rem}
 In fact, since the groups $H_*(\BdDiff_+(M); \bZ)$ are not in general finitely generated, the map
 \[
 H^*(\BdDiff_+(M); \bZ)\otimes \bQ\to H^*(\BdDiff_+(M); \bQ)
 \]
 could have a large kernel (see \cite[Theorem 0.9]{nariman2015stable} and \cite[Theorem 8.1]{morita1987characteristic}). 
 \end{rem}
 For flat $\bS^{2n}$-bundles, in contrast, we shall see that the Bott vanishing theorem (\cite{bott1970topological}) implies the following vanishing result.
 \begin{thm}\label{vanishing} The monomials in Pontryagin classes $p_1, \dots, p_n$ whose degrees are larger than $4n$ vanish in $H^*(\BdDiff_+(\bS^{2n});\bQ)$.
 \end{thm}
It is worth noting that \Cref{vanishing} is not an immediate consequence of the Bott vanishing theorem. The Bott vanishing theorem is about the vanishing of certain monomials in Pontryagin classes of the normal bundle of a foliation. For a flat $\bS^{2n}$-bundle $\pi\colon E\to B$, the vertical tangent bundle $T_{\pi}E$ is the normal of the foliation on $E$, so the Bott vanishing theorem provides the vanishing of certain monomials in Pontryagin classes of $T_{\pi}E$. We shall see in the proof of \Cref{vanishing} how the Pontryagin classes of  $T_{\pi}E$ are related to the Pontryagin classes of the sphere bundle $\pi\colon E\to B$.
% As we shall see our method can be used to prove other non-vanishing results for flat manifold bundles when we have a free torus action on the fiber.  For example, Galatius, Grigoriev, and Randal-Williams proved in \cite[Theorem 4.1 (ii)]{galatius2015tautological} that there exists an $\text{SO}(n)\times \text{SO}(n)$-action on higher dimensional analog of surfaces $\text{W}^n_g=\#_g \bS^n \times \bS^n$ which is the connected sum of $g$ copies of $\bS^n\times \bS^n$. They used this action to detect the non-vanishing of certain MMM-classes $\kappa_{ep_{i}}$   for all $i\in \{1,2,\dots, n-1\}$ and $g>1$ when $n$ is odd. We can use the same method to prove the following non-vanishing result.
% \begin{cor}\label{W}
% The powers of the classes $\kappa_{ep_{i}}$   for all $i\in \{1,2,\dots, n-1\}$ and $g>1$ are non-trivial in $H^*(\BdDiff(\text{W}^n_g);\bQ)$ when $n$ is odd. 
% \end{cor}
% \begin{rem}
% For the case of $n=1$ where $\text{W}^1_g$ is a closed genus $g$ surface,  MMM-classes $\kappa_{e^{i+1}}$ are denoted by $\kappa_i$. Kotschick and Morita (\cite{kotschick2005signatures}) showed that $\kappa_1^k$ is non-trivial in $H^{2k}(\BdDiff(\text{W}^1_g);\bQ)$ for all positive integer $k$ provided that $g\geq 3k$. The class $\kappa_2$ is not known to be non-trivial for flat surface bundles and by Bott vanishing (\cite[Theorem 8.1]{morita1987characteristic}) the classes $\kappa_i$ for $i>2$ vanish in $H^{*}(\BdDiff(\text{W}^1_g);\bQ)$.
% \end{rem}
\subsection*{Making bundles flat up to bordism} Our method to prove \Cref{oddspheres} is motivated by the following question in foliation theory. Suppose for an $n$-dimensional manifold $M$, we have a smooth fiber bundle $M\to E_0\to B_0$ whose fiber is $M$ and the base is a closed compact manifold. We want to see whether we can change the fiber bundle "up to bordism" to put a flat structure on its total space, meaning to put a codimension $n$  foliation on the total space that is transverse to the fibers.   More precisely, we want to find a bordism $W$ whose boundary $\partial W$ is the disjoint union $B_0\coprod B_1$ and a fiber bundle $M\to E\to W$ over the bordism such that its restriction to $B_0$ is given by $E_0$ and its restriction to $B_1$ is a foliated bundle. We use the equivariant version of Mather-Thurston's theory (\cite[Section 1.2.2]{nariman2015stable} and \cite[Section 5.1]{nariman2014homologicalstability}) to answer this question in the following two cases. 
 \begin{thm}\label{G}
 Let $G$ be a finite-dimensional connected Lie group. Then any principal $G$-bundle over a closed manifold is cobordant via $G$-bundles to a foliated $G$-bundle (not necessarily flat principal $G$-bundle).
 \end{thm}
 We shall see that for the case $G=\text{SU}_2$ which is diffeomorphic to $\bS^3$, this implies that the powers of the Euler class and the first Pontryagin class are non-trivial in $H^*(\BdDiff_+(\text{SU}_2);\bQ)$.
 \begin{thm}\label{T}Suppose $M$ is a manifold with a free torus $T$ action. Let $M\to E\xrightarrow{p} B$ be a $M$-bundle over a closed manifold $B$ that is classified by a map $B\to\mathrm{B}T$. Then this $M$-bundle is cobordant to a foliated $M$-bundle.
 \end{thm}
 We shall see that this theorem implies the non-vanishing results in \Cref{oddspheres}.
 \begin{rem}
 Since the techniques also work for volume-preserving diffeomorphisms, in \Cref{G} and \Cref{T}, we can arrange the foliated bundle at the other end of the bordism to have volume-preserving holonomies.
 \end{rem}
 \subsection*{The difference between flat Euclidean bundles and flat sphere bundles} Recall that the Euler class and topological Pontryagin classes are defined for $C^0$-Euclidean bundles since they are classes in $H^*(\mathrm{BHomeo}_+(\bR^{d}); \bQ)$. So one can ask about the vanishing or non-vanishing of monomials in such classes for $C^r$-flat Euclidean bundles for $r\geq 0$. For $r=0$, it is a consequence of McDuff's theorem (\cite[Section 2, Theorem 2.5]{mcduff1980homology}) that the map 
 \[
 \mathrm{BHomeo}^{\delta}_+(\bR^{d})\to  \mathrm{BHomeo}_+(\bR^{d})
 \]
 induces a homology isomorphism. In low dimensions, we know that the homotopy types are as follows: $\mathrm{Homeo}_+(\bR^3)\simeq \text{SO}(3)$ and $\mathrm{Homeo}_+(\bR^2)\simeq \text{SO}(2)$ (\cite[Theorem 1.2.3]{MR0334262} and \cite{Balcerak1981HomotopyTO}). 
 Therefore,  for Euclidean bundles in low dimensions, the natural maps $\bQ[p_1]\to H^*(\mathrm{BHomeo}_+(\bR^3);\bQ)$ and $\bQ[e]\to H^*(\mathrm{BHomeo}_+(\bR^2);\bQ)$ induce isomorphisms. In higher dimensions, we have the result of Galatius--Randal-Williams (\cite{galatius2022algebraic}) that the map
 \[
 \bQ[e,p_1, p_2, \dots]\to H^*(\mathrm{BHomeo}_+(\bR^{2n});\bQ)
 \]
 is injective for $2n\geq 6$ and in the odd-dimensional case, from \cite[Corollary 1.2]{galatius2022algebraic} we have that the map
  \[
 \bQ[p_1, p_2, \dots]\to H^*(\mathrm{BHomeo}_+(\bR^{2n-1});\bQ)
 \]
 is injective for $2n-1\geq 7$. Combining these results with McDuff's theorem, we obtain that for $C^0$-{\it flat} $\bR^{2n}$-bundles, all classes in $ \bQ[e,p_1, p_2, \dots]$ are non-trivial when $2n\geq 6$ and when $n=1$, all powers of the Euler class are non-trivial. For $C^0$-{\it flat} $\bR^{2n-1}$-bundles, all classes in $ \bQ[p_1, p_2, \dots]$ are non-trivial when $2n-1\geq 7$ and when $n=2$, all powers of the first Pontryagin class are non-trivial.
 
 However, for higher regularities, the situation is different. Let $\mathrm{Diff}_+^{r}(\bR^n)$ denote $C^r$ orientation-preserving diffeomorphisms of $\bR^n$. It is a deep result of Segal (\cite[Prop.~1.3 and~3.1]{segal1978classifying}),  that there is a map $$\mathrm{BDiff}_+^{r,\delta}(\bR^n)\to \mathrm{BS}\Gamma^{r}_{n}$$ that induces a homology isomorphism, where $\mathrm{BS}\Gamma^{r}_{n}$ is the classifying space of Haefliger structures for codimension $n$ foliations that are transversely oriented (see \cite[Section~1]{segal1978classifying} for more details).  There is also a map 
\[
\nu\colon \mathrm{BS}\Gamma^{r}_{n}\to \mathrm{B}\text{GL}^+_n(\bR)
\]
which classifies oriented normal bundles to the codimension $n$ foliations where $\text{GL}^+_n(\bR)$ is the group of invertible matrices with positive determinant. For all regularities, Haefliger showed that the map $\nu$ is at least $(n+1)$-connected (\cite[Remark~1 of Section~II.6]{haefliger1971homotopy}). Thurston proved (\cite{thurston1974foliations}) that for $r=\infty$, the map $\nu$ is $(n+2)$-connected and Mather showed that the same holds for $r\neq n+1$ (\cite[Section 7]{MR0356129}).  Hence, in particular, the induced map 
\[
\tilde{\nu}^*\colon H^*(\mathrm{B}\text{GL}^+_n(\bR); \bQ)\to H^*(\mathrm{B}\drDiff_{+}(\bR^n); \bQ)
\]
is an isomorphism for $*\leq n$ for all regularities. So the Pontryagin classes $p_i$ for $i\leq n/4$  are non-trivial and so is the Euler class $e$ when $n$ is even for all $r$. Therefore, for $C^r$-flat $\bR^n$ bundles, monomials in Pontryagin classes and the Euler class (if $n$ is even) are non-trivial if the degree of the monomial is less than $n+1$. Conjecturally (\cite[Problem 14.5]{hurder2003foliation}), the map $\tilde{\nu}^*$ is injective for $*\leq 2n$. For $r>1$ and $*>2n$, we have the Bott vanishing theorem (\cite{bott1970topological}) that implies that the monomials in Pontryagin classes of degrees greater than $2n$ vanish. In particular, when $r>1$ we have $e^4=p_n^2=0$ in $H^{8n}(\mathrm{B}\drDiff_{+}(\bR^{2n}); \bQ)$ unlike the case of powers of the Euler class in $H^*(\mathrm{B}\drDiff_{+}(\bS^{2n-1}); \bQ)$.
 
 \subsection*{Prigge's calculation and a correction of Haefliger's claim} There is another perspective to characteristic classes of flat manifold bundles due to Gelfand and Fuks (\cite{haefliger1975cohomologie}). For the case of flat sphere bundles, let $\mathcal{L}_{\bS^n}$ denote the topological Lie algebra of smooth vector fields on $\bS^n$. For the case of $n=1$, there is a natural map
 \[
\Phi\colon H^*(\mathcal{L}_{\bS^1}, \text{so}(2))\to H^*(\BdDiff_+(\bS^1);\bR),
 \]
 where $H^*(\mathcal{L}_{\bS^1}, \text{so}(2))$ is the relative continuous Lie algebra cohomology (see \cite[page 44]{haefliger2010differential}). The Lie algebra cohomology $H^*(\mathcal{L}_{\bS^1}, \text{so}(2))$ has been calculated by Gelfand and Fuks and is isomorphic to $\bR[e, \text{gv}]/(e\cdot \text{gv}=0)$, where $e$ is the Euler class and $\text{gv}$ is also a degree $2$ class known as the Godbillon-Vey class. Morita (\cite[Theorem 1.1]{morita1984nontriviality}) showed that this map is injective. 

There is a general van Est type theorem (\cite[page 43]{haefliger2010differential}) that for the case of the group $\Diff_+(\bS^3)$ implies that $H^*(\mathcal{L}_{\bS^3}, \text{so}(4))$ is isomorphic to the smooth group cohomology $H^*_{\text{sm}}(\Diff_+(\bS^3);\bR)$.  In higher dimensions, it is expected from Bott's belief in \cite[page 217]{MR488080} that there is at least a map
\[
H^*_{\text{sm}}(\Diff_+(\bS^n);\bR)\to H^*(\mathcal{L}_{\bS^n}, \text{so}(n+1)).
\]
Haefliger in \cite[page 154]{MR501011} sketched his method using rational homotopy theory to prove a claim  that the kernel of the map 
  \[
 H^*(\mathrm{B}\text{SO}(n+1); \bR)\to H^*(\mathcal{L}_{\bS^n}, \text{so}(n+1))
 \]
 is generated by the monomials in Pontryagin classes $p_1, \dots, p_{\lfloor n/2\rfloor}$ whose degrees are larger than $2n$. So according to his claim, $p_1^2$ lies in the kernel of the map 
   \[
 H^*(\mathrm{B}\text{SO}(4); \bR)\to H^*(\mathcal{L}_{\bS^3}, \text{so}(4)),
 \]
  where the target is isomorphic to  $H^*_{\text{sm}}(\Diff_+(\bS^3);\bR)$ by the van Est theorem. So Haefliger's claim that $p_1^2$ vanishes in $H^*_{\text{sm}}(\Diff_+(\bS^3);\bR)$ would imply its vanishing also in $H^*(\BdDiff_+(\bS^{3});\bR)$. But Nils Prigge in an appendix in the earlier version of this paper which shall become a separate paper proved that Haefliger's vanishing only works when $n$ is even, and using his method Prigge showed that for $n=3$, in fact, the class $p_1^2$ is not zero in the smooth group cohomology of $\Diff_+(\bS^3)$ which is isomorphic to $H^*(\mathcal{L}_{\bS^3}, \text{so}(4))$.

 In a forthcoming paper \cite{Nils}, Prigge in particular proves that 
   \[
 H^*(\mathrm{B}\text{SO}(4); \bR)\to H^*(\mathcal{L}_{\bS^3}, \text{so}(4)),
 \]
 is injective, which corrects the claim by Haefliger (\cite[page 154]{MR501011}). In the course of the proof, he found new relations between characteristic classes in the Gelfand-Fuks cohomology $H^*(\mathcal{L}_{\bS^3}, \text{so}(4))$ which is isomorphic to the smooth group cohomology $H^*_{\text{sm}}(\Diff_+(\bS^3);\bR)$. Morita observed that these relations combined with our non-vanishing result give a new type of relations between secondary characteristic classes and characteristic classes in $H^*(\BdDiff_+(\bS^{3});\bR)$. 
 
 To describe such a relation, recall for a foliation $\mathcal{F}$ on $M$ with a trivial normal bundle, there are secondary classes in $H^*(M)$ coming from Gelfand-Fuks cohomology (see \cite{pittie1976characteristic}). For the universal flat trivial $\bS^3$-bundle $$\pi\colon \bS^3\times \overline{\BDiff_+(\bS^3)}\to  \overline{\BDiff_+(\bS^3)},$$ we have a codimension $3$ foliation on the total space and there is a characteristic class $h_2c_2$ in $H^7(\bS^3\times \overline{\BDiff_+(\bS^3)};\bR)$ (see \cite{pittie1976characteristic} and \cite{Nils} for the definition of these characteristic classes). If we integrate $h_2c_2$ along the fiber, we obtain a class $\int_{\pi}h_2\cdot c_2$ in $H^4(\overline{\BDiff_+(\bS^3)};\bR)$. There is a certain class $\bar{x}_2$ in $H^4_{\text{sm}}(\Diff_+(\bS^3);\bR)$ whose image in $H^4(\BdDiff_+(\bS^3);\bR)$ extends the class $\int_{\pi}h_2\cdot c_2$  to a class in $H^4(\BdDiff_+(\bS^3);\bR)$. We denote this extension of $\int_{\pi}h_2\cdot c_2$  also by $\bar{x}_2$. Morita observed that Prigge's calculation implies the equality
 \[
 p_1^2=-e\cdot \bar{x}_2
 \]
 in $H^*(\BdDiff_+(\bS^3);\bR)$. Since we showed that $p_1^2$ is non-trivial, the class $\bar{x}_2$ is also non-trivial. This relation is interesting because the secondary class $h_2c_2$ is a continuously varying class (see \cite[Remark 2.9]{MR0769761} where the notation is $y_2c_2$). So $\bar{x}_2$ is intrinsically a real-valued class but $p_1$ and $e$ are defined over rational numbers. We hope to pursue finding such relations for higher dimensional spheres.  
 \subsection*{Acknowledgments} I am first and foremost indebted to Shigeyuki Morita for his questions, comments, and corrections that led me to write this paper. I am grateful to S\o ren Galatius for his comments and suggestions to study the free torus action in the context of Mather-Thurston's theorem. I would like to thank Sander Kupers for sending me the reference \cite{MR2509703}. I also would like to heartily thank Nils Prigge's efforts to rectify Haefliger's calculations.  The author was partially supported by  NSF CAREER Grant DMS-2239106 and Simons Foundation Collaboration Grant (855209).
 
% NP would like to thank Sam Nariman for offering to write this appendix and asking about this very interesting question. This research was supported by the Knut and Alice Wallenberg Foundation through grant no. 2019.0519.
 \section{Equivariant Mather-Thurston's theorem} In this section, we recall from \cite[Section 1.2.2]{nariman2015stable} and \cite[Section 5.1]{nariman2014homologicalstability} the equivariant version of Mather-Thurston's theorem as the main tool in this paper. 
 
 Mather-Thurston's theorem (\cite{MR0356085,thurston1974foliations}) is an h-principle theorem in foliation theory that relates the homotopy type of the classifying space of  Haefliger space to the group homology of diffeomorphism groups. Since we are interested in orientation-preserving diffeomorphisms in this paper, we recall Mather-Thurston's theorem in this context. Let $M$ be an orientable smooth manifold and let $\Diff_+^r(M)$ denote the group of $C^r$ orientation-preserving diffeomorphisms of  $M$ with the $C^r$-Whitney topology. If we drop the regularity $r$, we mean smooth diffeomorphisms. We decorate it with superscript $\delta$  if we consider the same group with the discrete topology. The identity homomorphism $\Diff_+^{r,\delta}(M)\to \Diff_+^r(M)$ induces a map
   \begin{equation}\label{eta}
   \eta: \BDiff_+^{r,\delta}(M)\to \BDiff_+^r(M).
   \end{equation}

   Thurston in fact studied $\overline{\BDiff_+^r(M)}$ which is the homotopy fiber of the map $\eta$. This space classifies foliated trivial $M$-bundles. Consider a semi-simplicial model $\overline{\BDiff_+^r(M)}_{\bullet}$, where the set of $k$-simplicies is given by the set of foliations on the trivial bundle $\Delta^k\times M\to \Delta^k$ that are transverse to the fibers and whose holonomies lie in $\Diff^r_+(M)$. The (fat) realization  $||\overline{\BDiff_+^r(M)}_{\bullet}||$ is a model for $\overline{\BDiff_+^r(M)}$.
   
   Note that the simplicial group $\text{Sing}_{\bullet}(\Diff_+^r(M))$, which is the singular complex of the topological group $\Diff_+^r(M)$, acts levelwise on $\overline{\BDiff_+^r(M)}_{\bullet}$. Milnor's theorem (\cite{MR0084138}) implies that the group $||\text{Sing}_{\bullet}(\Diff_+^r(M))||$ is homotopy equivalent to $\Diff_+^r(M)$, and given that $\overline{\BDiff_+^r(M)}$ is the homotopy fiber of $\eta$, the homotopy quotient 
   \begin{equation}\label{bdiff}
||\overline{\BDiff^r_+(M)}_{\bullet}||\hcoker ||\text{Sing}_{\bullet}(\Diff^r_+(M))||
\end{equation}
is weakly equivalent to $\BDiff_+^{r,\delta}(M)$.

The space $\overline{\BDiff_+^r(M)}$ is the geometric part of Mather-Thurston's theorem. To recall the part that is more amenable to homotopy theoretic techniques, let $\mathrm{S}\Gamma_n^r$ denote the topological groupoid whose space of objects is $\bR^n$ and whose space of morphisms is given by germs of $C^r$ orientation-preserving diffeomorphisms between two points in $\bR^n$ with a sheaf topology (see \cite{haefliger1971homotopy}). There is a natural map
   \[
 \nu\colon  \mathrm{BS}\Gamma_n^r\to \mathrm{B}\text{GL}^+_n(\bR)
   \]
    induced by the derivative of germs. Let $\overline{\mathrm{B}\Gamma_n^r}$ denote the homotopy fiber of the map $\nu$. Let $\tau_M\colon M\to  \mathrm{B}\text{GL}^+_n(\bR)$ be the map that classifies the tangent bundle and let $\tau_M^*(\nu)$ be the bundle over $M$ induced by the pullback of the map $\nu$ via $\tau_M$.  Let $\text{Sect}(\tau_M^*(\nu))$ be the space sections of $\tau_M^*(\nu)$. One can find a model for $\text{Sect}(\tau_M^*(\nu))$ on which $\Diff_+^r(M)$ acts as follows.  A $\nu$-tangential-structure on the $n$-dimensional manifold  $M$ is a bundle map $\mathrm{T}M\to \nu^*\gamma_n$, where $\gamma_n$ is the tautological vector bundle over $ \mathrm{BGL}_n^+(\bR)$. We denote the space of $\nu$-tangential-structures on  $M$ by $\mathrm{Bun}(\mathrm{T}M, \nu^*\gamma_n)$ and equip it with the compact-open topology. Note that $\Diff_+^r(M)$ acts on  $\mathrm{Bun}(\mathrm{T}M, \nu^*\gamma_n)$ by precomposing with the differential of diffeomorphisms. In \cite[Section 1.2.2]{nariman2015stable}, we showed the following version of Mather-Thurston's theorem.

\begin{thm}[Equivariant Mather-Thurston] There is a semi-simplicial map
\[
\overline{\BDiff^r_+(M)}_{\bullet}\to \text{Sing}_{\bullet}(\mathrm{Bun}(\mathrm{T}M, \nu^*\gamma_n)),
\]
which is equivariant with respect to the $\text{Sing}_{\bullet}(\Diff_+^r(M))$-action and  induces an acyclic map between fat realizations. 
\end{thm}
\begin{rem}
Mather and Thurston (\cite{MR0356085}) first proved that this map is a homology isomorphism for compact manifolds and for compactly supported diffeomorphisms of open manifolds. Later McDuff proved (\cite{mcduff1980homology}) the non-compactly supported case for open manifolds. 
\end{rem}

By Milnor's theorem (\cite{MR95518}) and the fact that the fat realization and geometric realization of simplicial sets are weakly equivalent (\cite[Section 1.2]{MR3995026}), the natural map $|| \text{Sing}_{\bullet}(\mathrm{Bun}(\mathrm{T}M, \nu^*\gamma_n))||\to \mathrm{Bun}(\mathrm{T}M, \nu^*\gamma_n)$ is a weak equivalence. So after realization, we obtain a map
\[
\overline{\BDiff^r_+(M)}\to \mathrm{Bun}(\mathrm{T}M, \nu^*\gamma_n)
\]
that induces an acyclic map and is equivariant with respect to the map $||\text{Sing}_{\bullet}(\Diff_+^r(M))||\xrightarrow{\simeq} \Diff_+^r(M)$.
\begin{cor}\label{MT} The map $\eta$ in \eqref{eta}, factors as follows 
\[
\BDiff_+^{r,\delta}(M)\xrightarrow{\beta} \mathrm{Bun}(\mathrm{T}M, \nu^*\gamma_n)\hcoker \Diff^r_+(M)\to  \BDiff_+^r(M),
\]
where the map $\beta$ is an acyclic map.
\end{cor}
 McDuff also proved the volume-preserving case of Mather-Thurston's theorem (\cite{MR707329,MR699012, mcduff1983local}, \cite[Section 2.2]{nariman2016moduli}). Suppose that $M$ is a compact manifold with a fixed volume form. Let $\Gamma_n^{\text{vol}}$ denote the topological groupoid whose space of objects is $\bR^n$ and whose space of morphisms is given by germs of volume-preserving diffeomorphisms of $\bR^n$. Now let $\theta\colon \mathrm{B}\Gamma_n^{\text{vol}}\to \mathrm{B}\text{SL}_n(\bR)$ and also let $\gamma_n$ be the tautological vector bundle over $ \mathrm{BSL}_n(\bR)$. Similarly we can define $\overline{\BDiff_{\text{vol}}(M)}$ and $\mathrm{Bun}(\mathrm{T}M, \theta^*\gamma_n)$. Then the equivariant version of McDuff's theorem (\cite[Section 2.2]{nariman2016moduli}) says that there is a semi-simplicial map
\[
\overline{\BDiff_{\text{vol}}(M)}_{\bullet}\to \text{Sing}_{\bullet}(\mathrm{Bun}(\mathrm{T}M, \theta^*\gamma_n)),
\]
which is equivariant with respect to the $\text{Sing}_{\bullet}(\Diff_{\text{vol}}^r(M))$-action and induces an acyclic map between fat realizations to {\it the connected component} that it hits.

\section{Making a bundle flat up to bordism} In this section, we prove the main theorems \ref{G} and \ref{T} and then we show how \Cref{T} implies \Cref{oddspheres}.
\begin{proof}[Proof of \Cref{G}] We first translate the problem into finding a ``homological" section as follows. A principal $G$-bundle is classified by a map to the classifying space $\mathrm{B}G$. Forgetting the principal structure and just considering it as a $G$-smooth fibration induces a map 
\[
\alpha\colon \mathrm{B}G\to \BDiff_+(G).
\]
Here we assume that the map $\alpha$ is induced by the inclusion $G\to \Diff_+(G)$ that is given by the {\it left} action of $G$ on itself. 

To obtain a foliated $G$-bundle, we need to lift the map $\alpha$ to $\BdDiff_+(G)$. This is not always possible but we show that it is possible to ``homologically'' lift $\alpha$, meaning that we find the dotted line below to make the following diagram commute up to homotopy
\begin{equation}\label{l}
\begin{gathered}
 \begin{tikzpicture}[node distance=1.8cm, auto]
  \node (A) {$\mathrm{B}G$};
  \node (B) [right of=A, node distance=3cm] {$ \BDiff_+(G).$};
  \node (C) [above of= B ] {$\mathrm{Bun}(\mathrm{T}G, \nu^*\gamma_n)\hcoker \Diff_+(G)$};  
   \draw [->] (A) to node {$\alpha$}(B);
  \draw [->] (C) to node {$$}(B);
  \draw [->, dotted] (A) to node {$$}(C);
\end{tikzpicture}
\end{gathered}
\end{equation}
Then, \Cref{MT} would imply that we have a commutative diagram between oriented bordism groups
\[
 \begin{tikzpicture}[node distance=1.8cm, auto]
  \node (A) {$\Omega^{\text{SO}}_*(\mathrm{B}G)$};
  \node (B) [right of=A, node distance=3cm] {$\Omega^{\text{SO}}_*(\BDiff_+(G)),$};
  \node (C) [above of= B ] {$\Omega^{\text{SO}}_*(\BdDiff_+(G))$};  
   \draw [->] (A) to node {$\alpha_*$}(B);
  \draw [->] (C) to node {$$}(B);
  \draw [->] (A) to node {$$}(C);
\end{tikzpicture}
\]
which in turn implies the bordism statement in \Cref{G}. 

Suppose that $G$ is $n$-dimensional. Since the tangent bundle $\mathrm{T}G$ is trivial, the space of bundle maps $\mathrm{Bun}(\mathrm{T}G, \nu^*\gamma_n)$ is homotopy equivalent to the space of maps $\text{Map}(G, \overline{\mathrm{B}\Gamma_n})$. This homotopy equivalence, however, is not $\Diff_+(G)$-equivariant. But note that the trivialization $G\times \bR^n\to TG$ is $G$-equivariant with respect to the left action of $G$ on itself, the trivial action on $\bR^n$, and the action on $\mathrm{T}G$ induced as a subgroup of $\Diff(G)$ acting from the left on  $\mathrm{T}G$. Recall the trivialization sends $(g,v)$ to $(g, Dg(v))$, where $Dg$ is the differential of the left action by $g$ and it is easy to see that this map is $G$-equivariant. For example, consider $(h, w)$ in $\mathrm{T}G$ and $g$ in $G$. Then $g.(h, w)=(gh, Dg(w))$. Note that $(h, w)$ comes from $(h, D(h^{-1})(w))$ under the isomorphism $G\times \bR^n\to TG$. When $g$ acts on $(h, D(h^{-1})(w))$ as an element in $G\times \bR^n$, we obtain $(gh, D(h^{-1})(w))$ which maps to $(gh, Dg(w))$ via the isomorphism  $G\times \bR^n\to TG$. Hence, this isomorphism is $G$-equivariant. So the natural map
\[
f\colon \text{Map}(G, \overline{\mathrm{B}\Gamma}_n)\to  \mathrm{Bun}(\mathrm{T}G, \nu^*\gamma_n)
\]
is equivariant with respect to the $G$-actions on both sides. Hence, we have a commuting diagram 

\[
   \begin{tikzpicture}[node distance=3cm, auto]
 % \node (A) {$\overline{\BDiff(S^3)}$};
  \node (B) [] {$\text{Map}(G, \overline{\mathrm{B}\Gamma}_n)\hcoker G$};
 % \node (C) [right of=A, ] {$\overline{\BDiff(S^3)}$};  
  \node (D) [right of=B, node distance=4cm] {$ \mathrm{Bun}(\mathrm{T}G, \nu^*\gamma_n)\hcoker \Diff_+(G)$};
    \node (E) [below of=B, node distance=1.5cm] {$\mathrm{B}G$};
  \node (F) [below of=D, node distance=1.5cm] {$\mathrm{B}\Diff_+(G).$};
%  \draw[->] (C) to node {$$} (D);
 % \draw [->] (A) to node {$$} (B);
%    \draw [->] (A) to node {$$} (C);
  \draw [->] (B) to node {$$} (D);
    \draw [->] (E) to node {$\alpha$} (F);
  \draw [->] (B) to node {$$} (E);
  \draw [->] (D) to node {$$} (F);
\end{tikzpicture}
 \]
Note that the action of $G$ on the mapping space has fixed points (constant maps), so the left map to $\mathrm{B}G$ has a section. Therefore, the map $\alpha$ can be lifted to $ \mathrm{Bun}(\mathrm{T}G, \nu^*\gamma_n)\hcoker \Diff_+(G)$. 
\end{proof}
Already \Cref{G} implies \Cref{oddspheres} for $\bS^3$ as follows. Note that $G=\bS^3$ is a Lie group and also by Hatcher's theorem $\BDiff_+(\bS^3)\simeq \mathrm{B}\text{SO}(4)$. So the action of $\bS^3$ on itself from the left induces the map $\alpha$ that on cohomology gives
\[
\alpha^*\colon  H^*(\BDiff_+(\bS^3);\bQ)\cong \bQ[e, p_1]\to H^*(\mathrm{B}\bS^3; \bQ)\cong \bQ[c_2],
\]
where $e, p_1, c_2$ are the universal Euler class, the first Pontryagin class, and the second Chern class respectively. So we have $\alpha^*(e)=c_2$ and $\alpha^*(p_1)=-2c_2$. Given the diagram \eqref{l} that $\alpha^*$ factors through the group $H^*(\BdDiff_0(\bS^3); \bQ)$, the powers of the Euler class and the first Pontryagin class map non-trivially to $H^*(\BdDiff_0(\bS^3); \bQ)$.
%\begin{rem} Morita pointed out to the author that by also considering the right action of $\bS^3$ on itself, we can see that $e^k$ and $p_1^k$ are linearly independent in $H^*(\BdDiff_0(\bS^3); \bQ)$ for $k$ odd. It would interesting to determine whether 
%\[
%\bQ[e,p_1]\to H^*(\BdDiff_0(\bS^3); \bQ),
%\]
%is injective. 
%\end{rem}
Now using the same idea, we are ready to prove \Cref{T}.
\begin{proof}[Proof of \Cref{T}]
Let $n$ be the dimension of $M$. For the free torus $T=(S^1)^r$ action on $M$, let $Q$ be the quotient $M/T$ and $\pi$ be the natural map $\pi\colon M\to Q$. Note that the vertical tangent bundle $T_{\pi}M$ is trivial by differentiating the torus action, so it is isomorphic to $M\times \mathfrak{t}$ where $\mathfrak{t}$ is the Lie algebra of $T$. Hence, there is a natural isomorphism between $TM$ and $\pi^*(TQ)\oplus \epsilon^r$, where $\epsilon$ is the trivial line bundle over $M$ and this isomorphism is $T$-equivariant with respect to the trivial action on the vectors in the bundle $\pi^*(TQ)\oplus \epsilon^r$. This is the consequence of the following claim.

\begin{claim} Let $G$ act freely on $M$. Then the vertical tangent bundle $T_{\pi}M$ of the quotient map $\pi\colon M\to M/G$ is isomorphic to the trivial bundle $M\times \mathfrak{g}$, where $\mathfrak{g}$ is the Lie algebra of $G$. This isomorphism is $G$-equivariant with respect to the adjoint action on $\mathfrak{g}$.
\end{claim}
To prove the claim, we shall  identify the fiber of the vertical tangent bundle $T_{\pi}(M)|_m$ at the point $m$ with the Lie algebra $\mathfrak{g}$ as follows. Let $\text{ev}_m\colon G\to M$ be the map that sends $m$ to $gm$. The derivative of this map at $g=\text{id}$ gives an isomorphism between $\mathfrak{g}$ and $T_{\pi}(M)|_m$. The derivative of the map induced by acting by an element $h\in G$ gives the map $D(h)\colon T_{\pi}(M)|_m\to T_{\pi}(M)|_{hm}$. Note that the composition $\mathfrak{g}\to T_{\pi}(M)|_m\to T_{\pi}(M)|_{hm}$ is the derivative of the map that sends $g$ to $hgm$ at $g=\text{id}$. Given that $hgm= (hgh^{-1}).hm$, we obtain a commutative diagram
\[
   \begin{tikzpicture}[node distance=3cm, auto]
 % \node (A) {$\overline{\BDiff(S^3)}$};
  \node (B) [] {$\mathfrak{g}$};
 % \node (C) [right of=A, ] {$\overline{\BDiff(S^3)}$};  
  \node (D) [right of=B, node distance=3cm] {$ T_{\pi}(M)|_m$};
    \node (E) [below of=B, node distance=1.5cm] {$\mathfrak{g}$};
  \node (F) [below of=D, node distance=1.5cm] {$T_{\pi}(M)_{hm}.$};
%  \draw[->] (C) to node {$$} (D);
 % \draw [->] (A) to node {$$} (B);
%    \draw [->] (A) to node {$$} (C);
  \draw [->] (B) to node {$D(\text{ev}_m)$} (D);
    \draw [->] (E) to node {$D(\text{ev}_{hm})$} (F);
  \draw [<-] (E) to node {$\text{Ad}_h$} (B);
  \draw [->] (D) to node {$D(h)$} (F);
\end{tikzpicture}
\]
This gives the claimed equivariant trivialization of $ T_{\pi}(M)$. $\qed$

Similar to the proof of \Cref{G}, it is enough to show that the torus action of the space of bundle maps $\text{\textnormal{Bun}}(TM, \nu^*\gamma_n)$ has fixed points. We think of bundle maps as the space of lifts of the classifying map $\tau_M\colon M\to \mathrm{BGL}_n^+(\bR)$ of the tangent bundle to $\mathrm{BS}\Gamma_n$. Note that the classifying map $\tau_M$ factors as $M\xrightarrow{\pi}Q\to \mathrm{BGL}_n^+(\bR)$. Recall that the map $\nu\colon\mathrm{BS}\Gamma_n\to \mathrm{BGL}_n^+(\bR)$ is at least $(n+2)$-connected (\cite[theorem 2]{thurston1974foliations}) and the dimension of $Q$ is less than $n$. Therefore, by obstruction theory, the map $Q\to \mathrm{BGL}_n^+(\bR)$ lifts to $\mathrm{BS}\Gamma_n$. Such a lift gives an element in $\text{\textnormal{Bun}}(TQ\oplus \epsilon^r,\nu^*\gamma_n)$. 

On the other hand, the torus action on $\text{\textnormal{Bun}}(TQ\oplus \epsilon^r,\nu^*\gamma_n)$ is trivial since $T$ acts trivially on $Q$ and on the bundle $\epsilon^r$. By the claim,  there is a $T$-equivariant isomorphism between $TM$ and $\pi^*(TQ)\oplus \epsilon^r$ and as a consequence we obtain a $T$-equivariant map 
 \[
\pi^*\colon\text{\textnormal{Bun}}(TQ\oplus \epsilon^r,\nu^*\gamma_n)\to \text{\textnormal{Bun}}(TM, \nu^*\gamma_n).
 \]
Therefore, the image of $\pi^*$ in $\text{\textnormal{Bun}}(TM, \nu^*\gamma_n)$ which is non-trivial lies in the fixed points of the $T$-action. \end{proof}

\begin{rem}
In the case of volume-preserving diffeomorphisms, it is a consequence of Moser's trick (\cite[Corollary 1.5.4]{MR1445290}) that the inclusion $\Diff_{\text{vol}}(M)\hookrightarrow \Diff_+(M)$ is a homotopy equivalence. So the induced map on classifying spaces $\BDiff_{\text{vol}}(M)\to \BDiff_+(M)$ is also a homotopy equivalence. Given McDuff's version of Mather-Thurston's theorem for volume-preserving diffeomorphisms, we can do the above argument to obtain foliations whose holonomies preserve the volume form. 
\end{rem}
Let us now use \Cref{T} to prove \Cref{oddspheres}.
\begin{proof}[Proof of \Cref{oddspheres}] Consider the sphere $\bS^{2n-1}$ as the quotient $\text{U}(n)/\text{U}(n-1)$. We have a standard free $\text{U}(1)$-action on $\bS^{2n-1}$ as follows
\[
\text{U}(1)\xrightarrow{\Delta} \text{U}(1)^n\to \text{U}(n)\to\text{SO}(2n)\to \Diff_+(\bS^{2n-1}).
\]
Recall that the Euler class and Pontryagin class for $\bS^{2n-1}$-bundles are defined by taking the infinite cone of each fiber to obtain a topological $\bR^{2n}$-bundle. So we have the following maps between classifying spaces
\[
\mathrm{B}\text{U}(1)\xrightarrow{\Delta} \mathrm{B}\text{U}(1)^n\to \mathrm{B}\text{SO}(2n)\to \BDiff_+(\bS^{2n-1})\to \mathrm{BHomeo}_+(\bR^{2n}).
\]
Hence, the pull-back of the monomials in classes $e$ and $p_i$ in $H^*(\mathrm{BHomeo}_+(\bR^{2n}); \bQ)$ for $i\leq n-1$ to $H^*(\mathrm{B}\text{U}(1); \bQ)$ are non-trivial multiples of powers of the first Chern class $c_1\in H^2(\mathrm{B}\text{U}(1); \bQ)$.

On the other hand, similar to the diagram \eqref{l},  \Cref{T} gives the following homotopy commutative diagram

\begin{equation}\label{ll}
\begin{gathered}
 \begin{tikzpicture}[node distance=1.8cm, auto]
  \node (A) {$\mathrm{B}\text{U}(1)$};
  \node (B) [right of=A, node distance=3cm] {$ \BDiff_+(\bS^{2n-1}).$};
  \node (C) [above of= B ] {$\mathrm{Bun}(\mathrm{T}\bS^{2n-1}, \nu^*\gamma_{2n-1})\hcoker \Diff_+(\bS^{2n-1})$};  
   \draw [->] (A) to node {$\alpha$}(B);
  \draw [->] (C) to node {$$}(B);
  \draw [->] (A) to node {$$}(C);
\end{tikzpicture}
\end{gathered}
\end{equation}

Therefore, we know that the map
\[
H^*(\BDiff_+(\bS^{2n-1}); \bQ)\to H^*(\mathrm{B}\text{U}(1); \bQ)
\]
factors through $H^*(\BdDiff_+(\bS^{2n-1}); \bQ)$ which implies that all the powers of  the classes $e$ and $p_i$ for $i\leq n-1$ are non-trivial in $H^*(\BdDiff_+(\bS^{2n-1}); \bQ)$.
\end{proof}
\begin{rem}
This argument, in fact, proves the slightly more general result that all elements in $\bQ[e,p_1,...,p_{n-1}] $ that are not in the kernel of the map $$H^*( \mathrm{B}\text{SO}(2n); \bQ)\to H^*(\mathrm{B}\text{U}(1);\bQ)$$ are non-trivial in $H^*(\BdDiff_+(\bS^{2n-1}); \bQ)$.
\end{rem}
\begin{rem}
As we mentioned, Morita observed that the classes $e^k$ and monomials in Pontryagin classes are linearly independent when $k$ is odd. It would be interesting to determine whether the map
\[
\bQ[e,p_1,\cdots, p_{n-1}]\to H^*(\BdDiff_0(\bS^{2n-1}); \bQ)
\]
is injective or not. 

\end{rem}
\begin{rem}
One interesting example on which a torus does not act freely is the higher dimensional analog of surfaces $\text{W}^n_g=\#_g \bS^n \times \bS^n$ which is the connected sum of $g$ copies of $\bS^n\times \bS^n$. Galatius, Grigoriev, and Randal-Williams proved in \cite[Theorem 4.1 (ii)]{galatius2015tautological} that there exists an $\text{SO}(n)\times \text{SO}(n)$-action $\text{W}^n_g$. They used this action to detect the non-vanishing of certain MMM-classes $\kappa_{ep_{i}}$   for all $i\in \{1,2,\dots, n-1\}$ and $g>1$. As we observed in \cite[Theorem 6.3]{nariman2014homologicalstability}, one can use this action to show that the powers of the classes $\kappa_{ep_{i}}$   for all $i\in \{1,2,\dots, n-1\}$ and $g>1$ are non-trivial in $H^*(\BdDiff(\text{W}^n_g);\bZ)$. It would be interesting to see if the classes $\kappa_{ep_{i}}$ are also non-trivial in $H^*(\BdDiff(\text{W}^n_g);\bQ)$.
 For the case of $n=1$ where $\text{W}^1_g$ is a closed genus $g$ surface,  MMM-classes $\kappa_{e^{i+1}}$ are denoted by $\kappa_i$. Kotschick and Morita (\cite{kotschick2005signatures}) showed that $\kappa_1^k$ is non-trivial in $H^{2k}(\BdDiff(\text{W}^1_g);\bQ)$ for all positive integer $k$ provided that $g\geq 3k$. The class $\kappa_2$ is not known to be non-trivial for flat surface bundles and by Bott vanishing (\cite[Theorem 8.1]{morita1987characteristic}) the classes $\kappa_i$ for $i>2$ vanish in $H^{*}(\BdDiff(\text{W}^1_g);\bQ)$. 
\end{rem}
\section{Even dimensional spheres and the Bott vanishing theorem} Here we use the Bott vanishing theorem to prove \Cref{vanishing}. First, let us recall the statement from  \cite{bott1970topological}. Let $\mathcal{F}$ be a foliation on a closed manifold $E$ of codimension $q$. Then the subalgebra of $H^*(E;\bQ)$ generated by the rational Pontryagin classes $p_i(\nu(\mathcal{F}))$ of the normal bundle $\nu(\mathcal{F})$ of $\mathcal{F}$ is trivial in cohomological degree larger than $2q$. In particular,  a monomial $P(\nu(\mathcal{F})):=p_{i_1}(\nu(\mathcal{F}))\cdots p_{i_k}(\nu(\mathcal{F}))$ vanishes if $4(i_1+i_2+\cdots+i_k)>2q$.

Recall from the introduction that to any smooth oriented sphere bundle $\bS^{2n}\to E\xrightarrow{\pi} B$, we assign Pontryagin classes $p_i(\pi)\in H^*(B; \bQ)$ by taking the infinite cone of the fibers and considering the associated topological Euclidean fiber bundle. In order to prove the vanishing theorem \ref{vanishing} for a flat smooth oriented $\bS^{2n}$-bundle $E\xrightarrow{\pi} B$, we need to relate these Pontryagin classes in $H^*(B; \bQ)$ to the Pontryagin classes of the normal bundle of the foliation on the total space in $H^*(E; \bQ)$. 
\begin{lem}\label{Igusa} Let $\pi\colon E\xrightarrow{}B$ be a smooth oriented $\bS^{m}$-bundle and let $c(\pi)\colon C(E)\to B$ be a topological $\bR^{m+1}$-bundle obtained by taking infinite cones on each fiber induced by the composition of  maps  \eqref{cone}. Let $\mathrm{T}_{\pi}E$ be the vertical tangent bundle of the fiber bundle $\pi$ and let $\epsilon$ be the trivial $\bR$-bundle over $E$. Then we have an isomorphism of topological bundles
\[
\mathrm{T}_{\pi}E\oplus \epsilon \cong \pi^*(C(E)),
\]
as topological $\bR^{m+1}$-bundles.
\end{lem}
\begin{proof}
In the proof of \cite[Corollary 1.4]{MR2509703}, Igusa showed that if  $p\colon L\to X$ is a smooth oriented sphere bundle over a compact manifold $X$ and  $s\colon X\to L$ is a section, then $C(L)$,  the cone of $L$ induced by the composition of  maps \eqref{cone} in the introduction, is isomorphic to $s^*(\mathrm{T}_{p}L)\oplus \epsilon$ as topological Euclidean space fiber bundles.

We denote the canonical section of the sphere bundle $q\colon \pi^*(E)\to E$  by $s$. Note that $\pi^*(C(E))$ is isomorphic to $C(\pi^*(E))$. By Igusa's theorem, the fiber bundle $\pi^*(C(E))$ is isomorphic to $s^*(\mathrm{T}_{q}\pi^*(E))\oplus \epsilon$. On the other hand, $s^*(\mathrm{T}_{q}\pi^*(E))$ is isomorphic to $\mathrm{T}_{\pi}E$. Hence, his theorem implies that 
\[
\mathrm{T}_{\pi}E\oplus \epsilon \cong \pi^*(C(E)),
\]
as topological $\bR^{m+1}$-bundles.
\end{proof}
Now suppose that $\pi\colon E\xrightarrow{}B$ is a smooth oriented flat $\bS^{2n}$-bundle. The vertical tangent bundle $\mathrm{T}_{\pi}E$ is the normal of the foliation on $E$. By definition, the Pontryagin classes $p_i(\pi)$ are defined to be the Pontryagin classes of the infinite cone bundle $C(E)$. On the other hand, by \Cref{Igusa}, we have
\[
p_i(\mathrm{T}_{\pi}E)=\pi^*(p_i(C(E)).
\]
So for any monomial $P(\mathrm{T}_{\pi}E):=p_{i_1}(\mathrm{T}_{\pi}E)\cdots p_{i_k}(\mathrm{T}_{\pi}E)$, we have $P(\mathrm{T}_{\pi}E)=\pi^*(P(\pi))$, where $P(\pi)$ is the monomial $p_{i_1}(\pi)\cdots p_{i_k}(\pi)$. If we fiber integrate $ e(\mathrm{T}_{\pi}E)\cdot P(\mathrm{T}_{\pi}E)$, which is the MMM-class $\kappa_{eP}$, we obtain
\[
\pi_!(e(\mathrm{T}_{\pi}E)\cdot P(\mathrm{T}_{\pi}E))= \pi_!(e(\mathrm{T}_{\pi}E)\cdot \pi^*(P(\pi)))= 2\cdot P(\pi),
\]
since $\pi_!(e(\mathrm{T}_{\pi}E))=\chi(\bS^{2n})=2$. By the Bott vanishing theorem, we know that $P(\mathrm{T}_{\pi}E)=0$ if $\text{deg}(P(\mathrm{T}_{\pi}E))>4n$. Hence, $P(\pi)=0$ if $P$ is a monomial of Pontryagin classes $p_1(\pi), \dots, p_n(\pi)$ whose degree is larger than $4n$.
\appendix

 \bibliographystyle{alpha}
\bibliography{reference}
\end{document}